\documentclass[12pt,letterpaper]{amsart}
\usepackage{geometry}
\usepackage{mathpazo} 
\usepackage{hyperref}
\usepackage{amsmath,amssymb,amsfonts,amsthm,setspace}
\newtheorem{theorem}{Theorem}
\newtheorem{proposition}{Proposition}
\newtheorem{lemma}{Lemma}
\newtheorem{corollary}{Corollary}
\theoremstyle{remark}
\newtheorem{remark}{Remark}
\newcommand{\C}{\mathbb{C}}
\newcommand{\e}{\mathrm{e}}
\newcommand{\D}{\Omega}
\newcommand{\ep}{\varepsilon}
\newcommand{\re}{\text{Re}}
\newcommand{\Dc}{\overline{\Omega}}
\newcommand{\dbar}{\overline{\partial}}
\newcommand{\zb}{\overline{z}}
\onehalfspace
\title[Essential norm estimates for the $\dbar$-Neumann operator]{Essential 
	norm estimates for the $\dbar$-Neumann operator on convex domains and 
	worm domains}
\author{\v{Z}eljko \v{C}u\v{c}kovi\'c}
\author{S\"{o}nmez \c{S}ahuto\u{g}lu}
\email{Zeljko.Cuckovic@utoledo.edu, Sonmez.Sahutoglu@utoledo.edu}
\address{University of Toledo, Department of Mathematics \& Statistics, 
Toledo, OH 43606, USA}
\date{\today}
\begin{document}
	\begin{abstract}
		In the paper we give a lower estimate for the essential norm of the $\dbar$-Neumann 
		operator on convex domains and worm domains of Diederich and Forn\ae{}ss.  
	\end{abstract}
\maketitle

Let $\D$ be a bounded pseudoconvex domain in $\C^n$ and $b\D$ denote the
boundary of $\D$.  The space of square integrable $(0,q)$-forms on $\D$ is denoted  
by $L^2_{(0,q)}(\D)$ for $0\leq q\leq n$. In this paper we will only consider $(0,q)$-forms 
instead of $(p,q)$-forms because the theory is independent of $p$. The operator 
$\dbar:L^2_{(0,q)}(\D)\to L^2_{(0,q+1)}(\D)$ is a closed,  linear, and densely defined 
unbounded operator and it has a Hilbert space adjoint 
$\dbar^*:L^2_{(0,q+1)}(\D)\to L^2_{(0,q)}(\D)$.  This is an important operator in complex 
analysis. 
  
The $\dbar$-Neumann operator, denoted by $N_q$, is the solution 
operator for the complex Laplacian 
$\Box_q=\dbar\dbar^*+\dbar^*\dbar:L^2_{(0,q)}(\D)\to L^2_{(0,q)}(\D)$. 
The $\dbar$-Neumann operator is a self-adjoint bounded linear operator on 
$ L^2_{(0,q)}(\D)$ and $\dbar^*N_q$ gives the solution operator for $\dbar$ with 
minimal norm. Sobolev regularity properties of $N_q$ are important in several 
complex variables and have been widely studied. For a survey of such results we 
refer the reader to  \cite{BoasStraube99}. For more information about the 
$\dbar$-Neumann operator we refer the reader to two excellent books on the 
subject \cite{ChenShawBook,StraubeBook}. 

Compactness of the $\dbar$-Neumann operator is stronger than its global 
regularity \cite{KohnNirenberg65}. There are potential theoretic (Property $(P)$ 
of Catlin \cite{Catlin84} and Property $(\widetilde{P})$ of McNeal \cite{McNeal02}) 
as well as geometric (\cite{MunasingheStraube07,Straube08}) sufficient conditions 
for compactness. Yet, it is not clear if these conditions are also necessary in general. 
In case of convex domains compactness of $N_q$ is well understood. Fu and Straube in 
\cite{FuStraube98} showed that for $1\leq q\leq n$ the following conditions are 
equivalent: compactness of $N_q$, the domain satisfying Property $(P_q)$, 
absence of $q$-dimensional varieties in the boundary of the domain, and  
compactness of the commutators $[P_{q-1},\zb_j]$ for $1\leq j\leq n$ 
(here $P_{q-1}$ is the Bergman projection on $(0,q-1)$-forms). For more information 
about compactness of the $\dbar$-Neumann problem and related topics we refer 
the reader to the  survey \cite{FuStraube01} and the book \cite{StraubeBook}.

The aim of this paper is to quantify the failure of compactness of the $\dbar$-Neumann 
operator in terms of boundary geometry.  As far as we know this is the first attempt 
in that direction. 
   
Let $X$ and $Y$ be two normed linear spaces and $T:X\to Y$ be a  bounded linear operator. 
The \textit{essential norm} of $T$, denoted by $\|T\|_e$, is defined as  
\[\|T\|_e=\inf\{\|T-K\|:K:X\to Y \text{ is a compact operator}\}\]
where $\|.\|$ denotes the operator norm. 

The motivation for this paper came from a previous paper 
\cite{CuckovicSahutogluSubmitted} in which we studied the essential norm estimates 
for a Hankel operator $H_{\varphi}=[\varphi,P]$ in terms of the behavior of the symbol 
$\varphi$  on the discs in the boundary.  Compactness of 
the $\dbar$-Neumann operator is closely connected to compactness of Hankel operators 
(see \cite{CelikSahutoglu12,CelikSahutoglu14}). We note that, it is still unclear if 
compactness of $H_{\varphi}$ on $A^2(\D)$, the Bergman space on $\D$, for all 
$\varphi\in C(\Dc)$ is sufficient for compactness of $N$.  This is known as D'Angelo's 
question. 

The plan of the paper is as follows: In the next section we will state the main result,  
Theorem \ref{ThmCn},  establishing a lower bound for the essential norm of $N_q$ 
on convex domains in $\C^n$. Then we continue with a section devoted to 
Theorem \ref{ThmWorm}, an application of our techniques to get a lower bound for 
the essential norm of the $\dbar$-Neumann operator on the Diederich-Forn\ae{}ss 
type worm domains. Finally, in the last section we present some basic facts about 
the essential norms of operators, the Proposition \ref{PropCn}, and the proofs of 
Theorems \ref{ThmCn} and \ref{ThmWorm}.   

\section*{The Main Result}
Throughout this paper $\|f\|$  will denote the $L^2$ norm of the function $f$. 
When we want to emphasize the domain we will denote 
the $L^2$ norm on $\D$ by $\|.\|_{\D}$. Let us define 
$C^1_0(\Dc)$ to be the set of real-valued functions that  are 
$C^1$-smooth on $\Dc$ and vanish on $b\D$. Let us also define   
 \[\alpha_{\D}=\sup\left\{\frac{2\int_{\D}\chi(z)dV(z)}{\|\nabla \chi \|}: 
\chi\in C^1_0(\Dc) \text{ and }\chi\not\equiv 0 \right\}\] 
where $\nabla \chi$ denotes the (real) gradient of $\chi$.  Let $r=(r_1,\ldots,r_n)$. 
By $r>0$ (respectively $r\geq 0$) we mean $r_j>0$ (respectively $r_j\geq 0$)  
for $1\leq j\leq n$. We will denote the polydisc in $\C^n$ centered at $w$ 
with polyradius $r>0$ by $D(w,r)=\{z\in \C^n:|z_j-w_j|<r_j,1\leq j\leq n\}$.  
We use the convention  $D(w,0)=\{w\}$. We define $\beta_{D(w,0)}=0$ and 
	\[\beta_{D(w,r)}= \frac{\prod_{k=1}^n r_k}{\sqrt{\sum_{k=1}^n \frac{1}{r_k^2}}}\]  
if $r>0$.

We note that $\alpha_D$ is the square root of the \textit{torsional rigidity} of $D$ 
when $D$ is a simply connected domain in $\C$. Physically, torsional rigidity of 
$D\subset \C$ is proportional to the discharge of a viscous fluid flowing through 
a pipe with the cross section $D$ (see \cite[pg 103]{PolyaSzegoBook}).
 
\begin{theorem}\label{ThmCn}
	Let $\D$ be a bounded convex domain in $\C^n$ and $\tau_{\D}$  denote the diameter 
	of $\D$. Assume that $q_{\D}$ is the largest dimension of the (affine) analytic varieties  
	in $b\D$.  
		\begin{itemize}
			\item[i.] If $q\geq q_{\D}=0$ or $q>q_{\D}\geq 0$ then $\|N_q\|_e=0$.
			\item[ii.] If $1\leq  q\leq q_{\D}\leq n-1$ then 
			\[\|N_q\|_e\geq \frac{C(n,q_{\D})}{\tau_{\D}^{2q_{\D}}}
			\sup\Big\{\beta_{D(w,r)}^2:D(w,r)\text{ is }
			q_{\D}\text{-dimensional  polydisc in } b\D \text{ with } r\geq 0 \Big\}\]
			where 
			\[C(n,q_{\D})=\frac{(q_{\D}+1)^{2q_{\D}+2}(n-q_{\D})^{2n-2q_{\D}}}{(n+1)^{2n+2}}
			\frac{3^{q_{\D}-1}}{2^{2q_{\D}+1}}.\]
			\item[iii.] If $1\leq q\leq q_{\D}=n-1$ and $\D$ has $C^1$-smooth boundary then 
			\[\|N_q\|_e\geq 
				\frac{(n-1)!}{\pi^{n-1}\tau_{\D}^{2n-2}}\sup\left\{\alpha^2_M:M 
				\text{	is an affine } (n-1)\text{-dimensional variety in } b\D\right\}.\]	 
		\end{itemize}
\end{theorem}

Let $\mathbb{D}$ be the unit open disc in the complex plane and 
$\Delta(a,b)=\{a+b\xi:\xi\in\mathbb{D}\}$ where $a,b\in \C^n$. 
So $\Delta(a,b)$ is an (affine) disc in $\C^n$ centered at $a$ with radius $|b|$ and $\Delta(a, 
0)=\{a\}$. 
Then iii. in Theorem \ref{ThmCn} leads to the following corollary.

\begin{corollary}\label{CorC2}
 Let $\D$ be a bounded convex domain in $\C^2$ with $C^1$-smooth boundary. Then 
\[\|N_1\|_e \geq \frac{r_{b\D}^{4}}{2\tau_{\D}^2}\]
where  $r_{b\D}=\sup\{|b|:\Delta(a,b)\subset b\D\}$ 
and $\tau_{\D}$ denote the diameter of $\D$.
\end{corollary}

We note that the  inequality in the corollary above is due to the following 	
fact: for a convex domain $M\subset \C$ we have $\alpha_M\geq 
r_M\sqrt{\frac{V(M)}{2}}$ where $r_M$ denotes the radius of the largest circle  
contained in $M$ (see \cite[pg 99-100]{PolyaSzegoBook}). 
	
\begin{remark}
	The essential norm of a self-adjoint operator is related to the essential spectrum 
	(part of the spectrum that is the complement of the eigenvalues with finite 
	multiplicities) of the operator. More precisely, if $T$ is a self-adjoint operator and    
	$\sigma_e(T)$ denotes the essential spectrum $T$, then 
	$\|T\|_e=\sup\{|\lambda|: \lambda \in \sigma_e(T)\}$. Since $N_q$ is self-adjoint 
	our results give lower bound estimates for the radius of the essential spectrum 
	of $N_q$ in case the domain is bounded and convex.
\end{remark}

\section*{Application to Worm Domains}
Let us start by defining more general versions of  Diederich-Forn\ae{}ss worm domains. 
Let $r>1, \beta>0$, and 
\[\rho_{\beta,r}(z_1,z_2)=\left|z_1-\e^{i2\beta\log|z_2|}\right|-1
+\sigma(|z_2|^2-r^2)+\sigma(1-|z_2|^2)\]
where 
\[\sigma(t)=\begin{cases}M\e^{-1/t},&t>0\\0,&t\leq 0\end{cases}\] 
for $M>0$. 
Then for large enough $M$ the domains
\[\D_{\beta,r}=\left\{ (z_1,z_2)\in \C^2:\rho_{\beta,r}(z_1,z_2)<0\right\}\] 
are smooth bounded and pseudoconvex (see \cite[Proposition 1]{BarrettSahutoglu12}).  
These domains have a total winding of $2\beta\log r$ and contain the 
annulus $A_{r}=\{\xi\in \C:1<|\xi|<r\}$. 

Worm domains originally have been  constructed to show that some smooth bounded 
pseudoconvex domains do not have Stein neighborhood basis for their closures 
\cite{DiederichFornaess77}. However, they turned out to be a class of domains with 
irregular Bergman projections and $\dbar$-Neumann operators 
\cite{Barrett92,Christ96,KrantzPeloso08}. 
Now they are considered one of the important classes of domains in several complex 
variables. We choose to work on these domains rather than the original worm domains 
because we can decouple the winding numbers from the size of the annuli.   

In the next theorem we give a lower bound estimate for the essential norm of the 
$\dbar$-Neumann operator on worm domains defined above. 

\begin{theorem}\label{ThmWorm}
Let  $r>1$ and $\beta>0$. Then  the $\dbar$-Neumann operator on 
$\D_{\beta,r}$ has the following essential norm estimate 
\[\|N_1\|_e\geq \max\left\{\left(\frac{\eta^2+1}{2}-\frac{\eta^2-1}{2\log\eta}\right)
\frac{\pi-2\beta\log\eta}{\pi+2\beta\log\eta}:1<\eta <\min\{\e^{\pi/2\beta},r\}\right\}.\]
\end{theorem}
It is interesting that the estimate in Theorem \ref{ThmWorm} depends on the winding 
number as well as  the size of the annulus in the boundary. In contrast, the irregularity 
results of the $\dbar$-Neumann operator on the worm domains depend on the winding 
number only \cite{Barrett92,BarrettSahutoglu12}. 

\section*{Proofs}
We will need the following lemmas for the proof of the theorems. 

\begin{lemma}\label{LemAkaki}
Let $X$ and $Y$ be two Hilbert spaces and $T:X\to Y$ be a bounded linear 
operator. Then 
\[\|T\|_e^2=\|T^*\|^2_e=\|T^*T\|_e=\|TT^*\|_e.\]
\end{lemma}
\begin{proof}
Let us define $\widetilde{T}:X\oplus Y\to X\oplus Y$ by 
$\widetilde{T}(x,y)=(0,Tx)$. Then $\|T\|=\|\widetilde{T}\|$. First we 
will show that $\|T\|_e=\|\widetilde{T}\|_e$. Let $K:X\to Y$ be a linear 
compact operator. Then $\|\widetilde{T}-\widetilde{K}\|=\|T-K\|$ where the 
linear compact operator $\widetilde{K}:X\oplus Y\to X\oplus Y$ is defined 
by $\widetilde{K}(x,y)=(0,Kx)$. Hence taking infimum over $K$ implies that  
$\|\widetilde{T}\|_e\leq \|T\|_e$. 

To show the reverse inequality, let  $\pi_X$ and $\pi_Y$ denote the projections 
from $X\oplus Y$ onto  $X$ and $Y$, respectively. Let 
$\widetilde{K}:X\oplus Y\to X\oplus Y$ be a compact linear operator. Then the 
component operators $\widetilde{K}_1=\pi_X \widetilde{K}$ and 
$\widetilde{K}_2=\pi_Y\widetilde{K}$ are compact. 
Let us define $\widetilde{T}_2=\pi_Y\widetilde{T}$. That is, $\widetilde{T}_2(x,y)=Tx$. 
Then 
\begin{align*}
\|\widetilde{T}-\widetilde{K}\|^2=&\|\widetilde{K}_1\|^2
+\|\widetilde{T}_2-\widetilde{K}_2\|^2\\
\geq &\sup\left\{\|Tx-\widetilde{K}_2(x,y)\|^2:\|x\|^2+\|y\|^2\leq 1\right\}\\
\geq &  \sup\left\{\|Tx-\widetilde{K}_2(x,0)\|^2:\|x\|^2\leq 1\right\}\\
=&\|T-K\|^2
\end{align*}
where $K:X\to Y$ is a compact operator defined by $Kx=\widetilde{K}_2(x,0)$.
Taking infimum over $\widetilde{K}$ we get $\|\widetilde{T}\|_e\geq \|T\|_e$. 
Therefore, we showed that  
\begin{align}\label{Eqn2}
 \|T\|_e=\|\widetilde{T}\|_e.
\end{align}

We will continue the proof with computing  $\widetilde{T}^*$. Let 
$x,u\in X$ and $y,v\in Y$. Then 
\begin{align*}
 \langle\widetilde{T}^*(x,y),(u,v)\rangle =&\langle 
(x,y),\widetilde{T}(u,v)\rangle \\
=& \langle(x, y),(0,Tu)\rangle\\
=&\langle (T^*y,0),(u,v)\rangle.
\end{align*}
Hence  $\widetilde{T}^*(x,y)=(T^*y,0), \|T^*\|=\|\widetilde{T}^*\|$ 
and, as was done earlier in the proof, one can show that 
$\|T^*\|_e=\|\widetilde{T}^*\|_e$. Furthermore  
\[\widetilde{T}^*\widetilde{T}(x,y)=\widetilde{T}^*(0,Tx)=(T^*Tx,0)\] 
and 
\[\widetilde{T}\widetilde{T}^*(x,y)=\widetilde{T}(T^*y,0)=(0,TT^*y).\]
Therefore, 
\begin{align}\label{Eqn3}
\|T^*T\|_e=\|\widetilde{T}^*\widetilde{T}\|_e \text{ and } 
\|TT^*\|_e=\|\widetilde{T}\widetilde{T}^*\|_e. 
\end{align}
Finally the fact that the Calkin algebra on a Hilbert space is a $C^*$-algebra 
(see, for example, \cite[5.6 Theorem]{ConwayBook}) implies that 
\[\|\widetilde{T}\|^2_e
=\|\widetilde{T}^*\|^2_e=\|\widetilde{T}^*\widetilde{T}\|_e
=\|\widetilde{T}\widetilde{T}^*\|_e.\]
Now combining the equality above with equalities \eqref{Eqn2} and 
\eqref{Eqn3} 
we get 
\[\|T\|_e^2=\|T^*\|^2_e=\|T^*T\|_e=\|TT^*\|_e.\]
Hence the proof of the lemma is complete.
\end{proof}

\begin{remark}
We use the Lemma above to show that the lower estimate in 
\cite[Theorem 2]{CuckovicSahutogluSubmitted} is sharp, in case 
$\varphi(z_1,z_2)=\zb_1$. First one can show that 
$H_{\zb_1}^{\mathbb{D}^2}(z_1^jz_2^k)=H_{\zb_1}^{\mathbb{D}}(z_1^j)z_2^k$ for 
all $j,k\geq 0$ and 
$\langle H^{\mathbb{D}}_{\zb_1}z_1^j,
H^{\mathbb{D}}_{\zb_1}z_1^k\rangle_{\mathbb{D}} =0$ 
unless $j=k$ (the inner product is on $\mathbb{D}$).  We use this fact in the 
equality below. In the computations below we denote the domain as a subscript 
unless  it is $\mathbb{D}^2$.
\begin{align*}
\left\|H_{\zb_1}^{\mathbb{D}^2}\sum_{j,k= 0}^{\infty}a_{jk}z_1^jz_2^k\right\|^2
&=\sum_{j,k=0}^{\infty}|a_{jk}|^2\left\|H_{\zb_1}^{\mathbb{D}}z_1^j\right\|^2_{\mathbb{D}} 
\|z_2^k\|^2_{\mathbb{D}}\\
&\leq \frac{1}{2}\sum_{j,k=0}^{\infty}|a_{jk}|^2\|z_1^j\|^2_{\mathbb{D}} 
\|z_2^k\|^2_{\mathbb{D}}\\
&= \frac{1}{2}\left\|\sum_{j,k= 
0}^{\infty}a_{jk}z_1^jz_2^k\right\|^2. 
\end{align*}
In the last inequality we used the fact that 
$\left\|H_{\zb_1}^{\mathbb{D}}\right\|=1/\sqrt{2}$ (see 
\cite[Theorem 1]{OlsenReguera16}). Therefore, 
$\left\|H_{\zb_1}^{\mathbb{D}^2}\right\|_{\mathbb{D}^2}\leq 1/\sqrt{2}$. 

Next we will show that $\left\|H_{\zb_1}^{\mathbb{D}^2}\right\|_e\geq1/\sqrt{2}$. 
Since $\left(H_{\zb_1}^{\mathbb{D}}\right)^*H_{\zb_1}^{\mathbb{D}}$ is 
a self-adjoint compact operator and its norm equals 
$1/2$, there exists an eigenfunction $f\in A^2(\mathbb{D})$ such that
\[\left(H_{\zb_1}^{\mathbb{D}}\right)^*H_{\zb_1}^{\mathbb{D}}f(z_1)
=\frac{f(z_1)}{2}.\] 
Then for $k\geq 0$ we have 
\[\left(H_{\zb_1}^{\mathbb{D}^2}\right)^*H_{\zb_1}^{\mathbb{D}^2}(f(z_1)z_2^k)
=\left(\left(H_{\zb_1}^{\mathbb{D}}\right)^*H_{\zb_1}^{\mathbb{D}}(f(z_1)) 
\right)z_2^k =\frac{f(z_1)z_2^k}{2}.\] 
That is, $1/2$ is an eigenvalue for 
$\left(H_{\zb_1}^{\mathbb{D}^2}\right)^*H_{\zb_1}^{\mathbb{D}^2}$ with infinite 
multiplicity. Then  $1/2$ is in the essential spectrum of
 $\left(H_{\zb_1}^{\mathbb{D}^2}\right)^*H_{\zb_1}^{\mathbb{D}^2}$
 (see \cite[Chapter XI, 4.6 Proposition]{ConwayBookFA}) and  
\[\left\|\left(H_{\zb_1}^{\mathbb{D}^2}\right)^*H_{\zb_1}^{\mathbb{D}^2}
\right\|_e\geq \frac{1}{2}.\] 
So
\[\frac{1}{\sqrt{2}}\leq 
\sqrt{\left\|\left(H_{\zb_1}^{\mathbb{D}^2}\right)^*H_{\zb_1}^{\mathbb{D}^2}
\right\|_e}=  \left\|H_{\zb_1}^{\mathbb{D}^2}\right\|_e
\leq \left\|H_{\zb_1}^{\mathbb{D}^2}\right\|\leq \frac{1}{\sqrt{2}}.\]
Therefore, 
$\left\|H_{\zb_1}^{\mathbb{D}^2}\right\|=\left\|H_{\zb_1}^{\mathbb{D}^2}\right\|_e
=1/\sqrt{2}$.
\end{remark}

\begin{remark}
Let $\D$ be a bounded pseudoconvex domain in $\C^n$. 
Let us define the operator $M_{\dbar\varphi}:A^2(\D)\to L^2_{(0,1)}(\D)$ as 
$M_{\dbar\varphi}f= f \dbar\varphi$. We note that $\|M_{\dbar \zb_k}\|_e=1$ 
for $1\leq k\leq n$. This can be seen as follows: Let $\{f_j\}$ be an orthonormal 
basis of $A^2(\D)$. Then using the fact that compact operators turn weakly convergent 
sequences into convergent sequence we conclude that 
\[ \lim_{j\to \infty}\|M_{\dbar \zb_k}f_j-Kf_j\|=  \lim_{j\to \infty}\|M_{\dbar 
\zb_k}f_j\|= \lim_{j\to \infty}\|f_j\|=1\] 
 for any compact operator $K:A^2(\D)\to L^2_{(0,1)}(\D)$. 
Hence, $1\leq \|M_{\dbar \zb_k}\|_e\leq \|M_{\dbar \zb_k}\| = 1$.

Let $\dbar^*N_{1,a}$ denote the 
restriction of $\dbar^*N_1$ onto $A^2_{(0,1)}(\D)$, the $(0,1)$-forms with 
square integrable holomorphic coefficients. Then one can show that   
\[\|H_{\zb_k}\|_e\leq\|\dbar^*N_{1,a}\|_e\|M_{\dbar\zb_k}\|_e\] 
for $k=1,2,\ldots,n$. The fact that  $\|M_{\dbar \zb_k}\|_e=1$ implies that 
\[\|\dbar^*N_1\|_e\geq \|\dbar^*N_{1,a}\|_e\geq 
\max\left\{\|H_{\zb_k}\|_e: k=1,2,\ldots,n\right\}.\]
Therefore, in case $\D=\mathbb{D}^2$ we get (see Corollary \ref{CorPercolate}) 	
\[\|N_1\|_e=\|\dbar^*N_1\|_e^2\geq 
\left\|\dbar^*N_{1,a}\right\|_e^2\geq 
\left\|H_{\zb_1}\right\|_e^2=\frac{1}{2}.\] 
 Comparing this estimate to Siqi Fu's result in \cite[pg 729]{Fu07}  about the 
bottom of the spectrum of $\Box_1$ shows that our estimate is not sharp on the 
bidisc. Indeed, the bottom of the spectrum of $\Box_1$ on the bidisc is 
$j_{0,1}^2/4\approx 1.44576576$ where $j_{0,1}\approx 2.4048$  is the first 
positive zero of the Bessel function of order zero. So 
$\|N_1\|_e=4/j_{0,1}^2\approx 0.69>1/2$. 
 \end{remark}

\begin{lemma}\label{LemDirectSum}
 Let $X_1,X_2,Y_1,Y_2$ be Hilbert spaces, and $T_1:X_1\to Y_1$ and 
 $T_2:X_2\to Y_2$ be bounded linear operators. Then 
 $T_1\oplus T_2:X_1\oplus X_2\to Y_1\oplus Y_2$ satisfies the 
 following equality
\[\|T_1\oplus T_2\|_e=\max \{\|T_1\|_e,\|T_2\|_e\}.\]
\end{lemma}

\begin{proof}
Let  $T=T_1\oplus T_2$, and $\pi_1$ and $\pi_2$ denote the projections from 
$Y_1\oplus Y_2$ onto $Y_1$ and $Y_2$, respectively.  Assume that 
$K:X_1\oplus X_2\to Y_1\oplus Y_2$ is a compact operator. 
Then $K_1=\pi_1K|_{X_1}$ and $K_2=\pi_2 K|_{X_2}$ are compact and 
\[\|(T-K)(x_1,0)\|\geq \|(T_1-K_1)x_1\| \text{ and } 
\|(T-K)(0,x_2)\|\geq \|(T_2-K_2)x_2\| \]
for all $x_1\in X_1$ and $x_2\in X_2$. Then  
\[\|T-K\|\geq \max\{\|T_1-K_1\|,\|T_2-K_2\|\} \geq 
\max\{\|T_1\|_e,\|T_2\|_e\}.\] 
Then taking infimum over $K$ we get 
\[\|T\|_e\geq \max\{ \|T_1\|_e,\|T_2\|_e\}.\]
Next we will prove the converse. Let $K_1:X_1\to Y_1$ and $K_2:X_2\to Y_2$ 
be compact operators. Then $K=K_1\oplus K_2:X_1\oplus X_2 \to Y_1\oplus Y_2$ 
is compact and 
\begin{align*}
\|T\|^2_e\leq \|T-K\|^2=&\sup\{ \|(T-K)(x_1,x_2)\|^2:\|x_1\|^2+\|x_2\|^2=1\}\\
=&\sup\{ \|(T_1-K_1)x_1\|^2+\|(T_2-K_2)x_2\|^2:\|x_1\|^2+\|x_2\|^2=1\} \\
\leq &\sup\{\|T_1-K_1\|^2\|x_1\|^2+\|T_2-K_2\|^2(1-\|x_1\|^2):\|x_1\|\leq 1\}\\
 =&\max\{\|T_1-K_1\|^2,\|T_2-K_2\|^2\}.
 \end{align*} 
Taking infimum over $K_1$ and $K_2$ we get
\[\|T\|_e\leq \max\{\|T_1\|_e,\|T_2\|_e\}.\]
Therefore, $\|T\|_e= \max\{\|T_1\|_e,\|T_2\|_e\}$
and the proof of the lemma is complete.
\end{proof}

Now we will prove a more precise version of \cite[Lemma 1]{CelikSahutoglu14}. 
We note that $ K^2_{(0,q)}(\D)$ denotes the $\dbar$-closed $(0,q)$-forms on $\D$.
 
\begin{lemma}\label{LemEstimate}
Let $\D$ be a bounded pseudoconvex domain in $\C^n$ for $n\geq 2$ and 
$g \in K^2_{(0,q+1)}(\D)$ where $1\leq q\leq n-1$.  Then 
there exist $g_j\in K^2_{(0,q)}(\D)$ for $1\leq j\leq n $ such that 
\[g=\sum_{j=1}^{n} g_j\wedge d\zb_j \text{ and } 
\sum_{j=1}^{n}\|g_j\|^2\leq \|g\|^2.\]
\end{lemma}
\begin{proof} 
Let $1\leq q\leq n-1$ and 
\[f=\sideset{}{'}\sum_{|J|=q}f_{J} d\zb_J=\dbar^*N_{q+1}g\in L^2_{(0,q)}(\D).\]
The symbol $\sideset{}{'}\sum_{|J|=q}$ above denotes the summation over 
strictly increasing index $J$. That is, $J=j_1j_2\cdots j_q$ with 
$j_1<j_2<\cdots j_q$.   Let $\vee$ denote the adjoint of the exterior 
multiplication. That is, if $f$ is a $(0,q)$-form $d\zb_j \vee f$ is a 
$(0,q-1)$-form such that 
$\langle h\wedge d\zb_j,f \rangle =\langle h,d\zb_j\vee f \rangle$ for all 
$h\in L^2_{(0,q-1)}(\D)$. We define $f_j=d\zb_j\vee f$ for $1\leq j\leq n$. 
Then one can show that 
\[f_j=\sideset{}{'}\sum_{|I|=q-1}f_{jI}d\zb^I 
\text{ for } 1\leq j\leq n \text{ and } f=\frac{1}{q}\sum_{j=1}^nd\zb_j\wedge 
f_j.\]
Every $f_J$  appears in $q$ different $f_j$'s for $J=jI$. The decomposition 
above was observed by Jeffery McNeal and it has appeared in 
\cite[pg. 75]{StraubeBook}. Then 
 \[\dbar f_j=\sideset{}{'}\sum_{|I|=q-1}\dbar f_{jI}\wedge 
 d\zb^I=\sideset{}{'}\sum_{|J|=q}F_J^j d\zb^J\]
 where each $F_J^j$ is a sum of at most $q$ terms of the form 
 $\frac{\partial f_{jI}}{\partial \zb_k}$ because each term appears at most 
once. Now we use the fact that 
$(x_1+\cdots +x_q)^2 \leq q(x_1^2+\cdots +x_q^2)$ for 
real numbers $x_1,\ldots, x_q$ to conclude that  
\[\|\dbar f_j\|^2\leq q\sideset{}{'}\sum_{|I|=q-1}\sum_{k=1}^n
 \left\|\frac{\partial f_{jI}}{\partial \zb_k}\right\|^2\]
for all $k$'s.  We note that $q^2$ appears on the second equality below because 
each  $\frac{\partial f_{jI}}{\partial \zb_k}$  appears  $q$ many times as 
$\frac{\partial f_{J}}{\partial \zb_k}$. Then we use  
\cite[Corollary 2.13]{StraubeBook}  to get 
\begin{align*}
\sum_{j=1}^n\|\dbar f_j\|^2 
&\leq q\sum_{j=1}^n\sideset{}{'}\sum_{|I|=q-1} \sum_{k=1}^n
 \left\|\frac{\partial f_{jI}}{\partial \zb_k}\right\|^2\\
&=q^2\sideset{}{'}\sum_{|J|=q} \sum_{k=1}^n
 \left\|\frac{\partial f_J}{\partial \zb_k}\right\|^2\\
& \leq q^2(\|\dbar f\|^2+\|\dbar^*f\|^2)\\
&=q^2\|g\|^2.
 \end{align*}
Let us define $g_j=\frac{(-1)^{q-1}}{q}\dbar f_j$. 
Then $\sum_{j=1}^n \|g_j\|^2\leq \|g\|^2$ and 
\[g=\dbar \dbar^*N_{q+1}g
=\dbar f=\frac{(-1)^{q-1}}{q}\sum_{j=1}^n\dbar f_j\wedge d\zb_j
=\sum_{j=1}^n g_j\wedge d\zb_j.\]
Hence the proof of the lemma is complete. 
\end{proof}

\begin{lemma}\label{LemPercolation}
 Let $\D$ be a bounded pseudoconvex domain in $\C^n$ and 
 $1\leq q\leq n-1$. Then 
 \[\|\dbar^*N_{q+1}\|_e\leq \|\dbar^*N_{q}\|_e.\] 
\end{lemma}
\begin{proof}
 Let $K^2_{(0,q)}(\D)$ denote the $\dbar$-closed square integrable 
$(0,q)$-forms. Assume that $f\in K^2_{(0,q+1)}(\D)$. Then by Lemma  
\ref{LemEstimate} for $1\leq k\leq n$ there exists  $f_k\in 
K^2_{(0,q)}(\D)$ such that 
\[f=\sum_{k=1}^nf_k\wedge d\zb_k \text{ and } 
\sum_{k=1}^n\|f_k\|^2\leq \|f\|^2.\]  
Assume that $\alpha_q=\|\dbar^*N_q\|_e$. Then for $\ep>0$ there exists a 
compact operator $K_q^{\ep}$ on $L^2_{(0,q)}(\D)$ such that 
\[\|\dbar^*N_q-K_q^{\ep}\|<\alpha_q+\ep.\] 
Let us define 
\[S_{q+1}f=\sum_{k=1}^n\dbar^*N_q(f_k)\wedge d\zb_k \text{ and }
K_{q+1}^{\ep}f=\sum_{k=1}^nK_q^{\ep}(f_k)\wedge d\zb_k.\]
Then 
\[\|S_{q+1}f-K_{q+1}^{\ep}f\|^2
\leq\sum_{k=1}^n\|\dbar^*N_qf_k-K_{q}^{\ep} f_k\|^2
\leq\|\dbar^*N_q-K_{q}^{\ep}\|^2\sum_{k=1}^n\|f_k\|^2. \]
That is, $\|S_{q+1}-K_{q+1}^{\ep}\|\leq \|\dbar^*N_q-K_{q}^{\ep}\|$ as 
$\sum_{k=1}^n\|f_k\|^2\leq \|f\|^2$. We note that on $K^2_{(0,q+1)}(\D)$ 
we have 
\[\dbar^*N_{q+1}=(I-P_{q+1})S_{q+1}\] 
because $\dbar^*N_{q+1}$ is the canonical 
solution operator for $\dbar$ and $S_{q+1}$ is a solution operator. Then  
\[\dbar^*N_{q+1}-\widetilde{K}_{q+1}^{\ep}=(I-P_{q+1})(S_{q+1}-K_{q+1}^{\ep})\] 
where $\widetilde{K}_{q+1}^{\ep}=(I-P_{q+1})K_{q+1}^{\ep}$ is a 
compact operator. Hence  
\[ \|\dbar^*N_{q+1}-\widetilde{K}_{q+1}^{\ep}\|
\leq \|S_{q+1}-K_{q+1}^{\ep}\|\leq  \|\dbar^*N_q-K_{q}^{\ep}\| 
\leq \alpha_q+\ep.\]
We complete the proof of the lemma by letting $\ep\to 0$.
\end{proof}

\begin{corollary}\label{CorPercolate}
 Let $\D$ be a bounded pseudoconvex domain in $\C^n$ and 
 $1\leq q\leq n-1$. Then 
 \[\|N_{q+1}\|_e\leq \|N_{q}\|_e.\] 
 Furthermore, $ \|N_0\|_e= \|N_1\|_e=\|\dbar^*N_1\|_e^2.$
\end{corollary}
\begin{proof} 
Let $1\leq q\leq n-1$. 
We will use Range's formula \cite{Range84,FuStraube01}, 
	\[N_q=(\dbar^*N_q)^*(\dbar^*N_q)+(\dbar^*N_{q+1})(\dbar^*N_{q+1})^*.\]
We note that  $(\dbar^*N_q)^*(\dbar^*N_q)$ is a mapping on $Ker(\dbar)$ 
(as it maps  $Im(\dbar^*)$ to zero) and similarly $(\dbar^*N_{q+1})(\dbar^*N_{q+1})^*$ 
is  a mapping on $Im(\dbar^*)$ (because it maps  $Ker(\dbar)$ to zero). 
Then we use the fact that $L^2_{(0,q)}(\D)=Ker(\dbar)\oplus Im(\dbar^*)$  
together with  Lemmas \ref{LemAkaki}, \ref{LemDirectSum}, and \ref{LemPercolation} 
to conclude that  
\begin{align}\label{EqnRange}
\|N_q\|_e=&\max\{\|\dbar^*N_q\|_e^2,\|\dbar^*N_{q+1}\|_e^2\}
= \|\dbar^*N_q\|_e^2 \\ 
\nonumber \|N_{q+1}\|_e=&\max\{\|\dbar^*N_{q+1}\|_e^2,\|\dbar^*N_{q+2}\|_e^2\}
=\|\dbar^*N_{q+1}\|_e^2.
\end{align}
Then Lemma \ref{LemPercolation} again implies that 
 $\|N_{q+1}\|_e\leq  \|N_{q}\|_e.$
 
In case $q=0$, Range's formula is $N_0=\dbar^*N_1(\dbar^*N_1)^*$. 
Then using Lemma \ref{LemAkaki} again we get  
$\|N_0\|_e=\|\dbar^*N_1\|_e^2.$
Furthermore, 
\[\|N_1\|_e=\max\{\|\dbar^*N_1\|_e^2,\|\dbar^*N_2\|_e^2\}=\|\dbar^*N_1\|_e^2.\] 
Therefore, $\|N_0\|_e= \|N_1\|_e=\|\dbar^*N_1\|_e^2.$
\end{proof}

\begin{lemma}\label{LemDiscs}
 Let $\D$ be a convex domain in $\C^n$ for $n\geq 2, p\in b\D$. Assume 
 that $M_1$ and $M_2$ are two analytic varieties in the $b\D$ and 
 $p\in M_1\cap M_2$. Then the convex hull of $M_1\cup M_2$ is an affine 
 analytic variety in $b\D$.
\end{lemma}
\begin{proof}
We will use the following fact: Any analytic variety in the boundary of a 
convex  domain in $\C^n$ is contained in affine analytic variety in the 
boundary of the domain. This fact was proven for analytic discs in 
\cite[Lemma 2]{CuckovicSahutoglu09} (see also \cite{FuStraube98} 
as well as \cite{BoasStraube92,McNeal92}). The same 
proof works for higher dimensional varieties in the boundary of a convex domain 
in $\C^n$ as well. Without loss of generality assume that 
\begin{itemize}
 \item[i.] $\D\subset \{(z_1,\ldots,z_n)\in \C^n:\re(z_n)< 0\}$ and $0\in b\D$,
 \item[ii.] $M_1$ and $M_2$ are affine analytic varieties in $b\D$ 
such that $0\in M_1\cap M_2$.
\end{itemize} 
Since $M_1$ and $M_2$ are  affine analytic varieties (using the fact that
tangent space of a complex manifold in $b\D$ is in the complex tangent space of 
$b\D$) we conclude that 
\[M_1\cup M_2\subset \{(z_1,\ldots,z_n)\in \C^n:z_n=0\}.\] 
Let $M$ denote the convex hull of $M_1\cup M_2$. Then $M$ is contained 
in the complex hyperplane $\{(z_1,\ldots,z_n)\in \C^n:z_n=0\}$. If $M$ is not 
an $(n-1)$-dimensional complex manifold, by applying rotation in the first 
$(n-1)$-variables if necessary, we may assume that $M\subset 
\{(z_1,\ldots,z_n)\in 
\C^n:\re(z_{n-1})=z_n=0\}$.  Invoking the fact that $M_1$ and $M_2$ are complex 
manifolds again we conclude that 
$M \subset \{(z_1,\ldots,z_n)\in \C^n:z_{n-1}=z_n=0\}$.  Using this 
argument, finitely many times, we reach the conclusion that $M$ is an 
affine analytic variety in $b\D$. 
\end{proof}

\begin{lemma}\label{LemWeakConv}
	Let $\D$ be a smooth bounded pseudoconvex or a  bounded convex domain 
	in $\C^n, p\in b\D,$ and $k_z(w)=K(w,z)/\sqrt{K(z,z)}$ where  $K$ is the 
	Bergman kernel of $\D$. Then  $k_z\to 0$ weakly as $z\to p$.
\end{lemma}
\begin{proof}
	Without loss of generality we may assume that $0\in \D$. Then 
	$A^{\infty}(\Dc)$, the space of functions holomorphic on $\D$ and smooth 
	up to the boundary, is dense in $A^2(\D)$. In case of bounded convex domain this 
	can be seen as follows:  if  $f\in A^2(\D)$ then the 
	function $f_{\delta}(z)=f((1-\delta)z)$ is holomorphic on a neighborhood of $\Dc$ 
	for any $\delta>0$ and $f_{\delta}\to f$ in $L^2$ norm as $\delta\to 0^+$.   In case 
	$\D$ is smooth bounded and pseudoconvex this is a result of 
	Catlin \cite[Theorem 3.2.1]{Catlin80}. 
	
Let $\ep>0$ be given. Then there exists $f_{\delta}\in A^{\infty}(\Dc)$ such that 
$\|f-f_{\delta}\|<\ep$. Then  
 \[|\langle f,k_z\rangle| \leq |\langle f-f_{\delta},k_z\rangle|+|\langle 
	f_{\delta},k_z\rangle|\leq \|f-f_{\delta}\|+|\langle f_{\delta},k_z\rangle|
	<\ep+|\langle f_{\delta},k_z\rangle|\]
However, we note that $\langle f_{\delta},k_z\rangle=f_{\delta}(z)/\sqrt{K(z,z)}\to 0$ 
as $z\to p$ because $K(z,z)\to \infty$ as $z\to p$ (see  
\cite[Theorem 6.1.17]{JarnickiPflugBook} and \cite{Pflug75}) and $f_{\delta}$ is bounded. 
Since $\ep$ was arbitrary we conclude that $\lim_{z\to p}\langle f,k_z \rangle= 0$ for 
any $f\in A^2(\D)$. That is, $k_z\to 0$ weakly as $z\to p$.
\end{proof}

\begin{proposition}\label{PropCn}
	Let $\D$ be a bounded convex domain in $\C^n$ and $\tau_{\D}$  denote the diameter 
	of $\D$. Assume that $b\D$ contains a non-trivial analytic variety and  $q$ is the largest 
	dimension of the analytic varieties in $b\D$. 
	\begin{itemize}
		\item[i.] If $1\leq q\leq n-1$ then 
		\[\|\dbar^*N_q\|_e \geq \frac{c(n,q)}{\tau_{\D}^{q}}
		\sup\left\{\beta_{D(w,r)}:D(w,r)\text{ is }
		q\text{-dimensional  polydisc in } b\D \text{ with } r\geq 0 \right\}\]
		where 
		\[c(n,q)=\frac{(q+1)^{q+1}(n-q)^{n-q}}{(n+1)^{n+1}}
		\left(\frac{3^{q-1}}{2^{2q+1}}\right)^{1/2}.\] 
		\item[ii.] If $q=n-1$ and $\D$ has $C^1$-smooth boundary then  
		\[\|\dbar^*N_{n-1}\|_e \geq 
		\sqrt{\frac{(n-1)!}{\pi^{n-1}}} \frac{1}{\tau_{\D}^{n-1}}\sup\left\{\alpha_M:M 
		\text{	is an affine } (n-1)\text{-dimensional variety in } b\D\right\}.\]		 
	\end{itemize}
\end{proposition}

\begin{remark}
	Let $r=(r_1,\ldots,r_n)$ with $r_j>0$ for $1\leq j\leq n$ and  
	$D(0,r)=\{z\in \C^n:|z_j|<r_j,1\leq j\leq n\} $ be a  polydisc. Using 
\eqref{Eqn6} in the proof of Proposition \ref{PropCn} one  
	can estimate 
	\[\alpha_{D(0,r)}\geq \sqrt{\frac{3^{n-1}\pi^n}{2^{2n-1}}} 
	\frac{\prod_{k=1}^n r_k}{\sqrt{\sum_{k=1}^n \frac{1}{r_k^2}}}.\] 
\end{remark}

\begin{proof}[Proof of Proposition \ref{PropCn}]
Let us prove i. first. Assume that $1\leq q\leq n-1$ is the largest dimension of 
analytic varieties in the boundary of $\D$. Then Lemma \ref{LemDiscs} implies 
that there is an affine $q$-dimensional analytic variety in $b\D$. By applying a 
holomorphic affine transformation if necessary, we may assume that 
\begin{itemize}
 \item[a.] $b\D$ contains a nontrivial $q$-dimensional polydisc 
$M=D(0,r)$ where $1\leq q\leq n-1$ and $r=(r_1,\ldots, r_q)$ for 
$r_j>0$ for all $j$, 
\item[b.] $M\times\{0\}\subset b\D\cap \{z\in \C^n:z_{q+1}=\cdots =z_n=0\}$,
\item[c.] there are no analytic discs in $b\D$ transversal to $M$. 
\end{itemize}
 Let $M_{\lambda}=\lambda M\subset M$ for $0<\lambda<1$ and let us 
denote  $z'=(z_1,\ldots, z_q)$ and $z''=(z_{q+1},\ldots,z_n)$ for $z=(z_1,\ldots, z_n)$.  
Assume that 
\[\D\subset \{z'\in \C^{q}:\|z'\|<\tau_{\D}\}\times \{z''\in \C^{n-q}:\|z''\|< 
\tau_{\D}, \re(z_n)>0\}.\] 
Let $\{p_j\}\subset \D_s=\{z''\in \C^{n-q}:(0,z'')\in \D\}$ be a sequence (to be 
determined later) converging to the origin  and 
\[\widetilde{f}_j(z'')=\frac{K_{\D_s}(z'',p_j)}{\sqrt{K_{\D_s}(p_j,p_j)}}.\] 
Then $\|\widetilde{f}_j\|_{\D_s}=1$ and Lemma \ref{LemWeakConv} implies that 
the sequence  $\{f_j\}$ converges to zero weakly as $p_j\to 0 \in b\D_s$. One can show 
that convexity of $\D$ implies that $M_{\lambda}\times (1-\lambda)\D_s\subset \D$.

We use \cite[Theorem 1]{Blocki13} (a version of Ohsawa-Takegoshi Theorem 
\cite{OhsawaTakegoshi87}) repeatedly $q$ times  with $D= \{z\in \C:|z|<\tau_{\D}\}$ 
and the fact that $c_D(0)=\sqrt{\pi K_D(0)}=1/\tau_{\D}$ to extend 
$\widetilde{f}_j$'s to $\D$, we call the extension $f_j$, such that  
\begin{align}\label{Eqn4}
\|f_j\|^2_{\D}\leq \pi^q\tau_{\D}^{2q}
\end{align} 
 and $f_j(0,z'')=\widetilde{f}_j(z'')$. Let
\[\chi_j(\xi)=\frac{2}{\pi 
\lambda^2r_j^2}\left(1-\frac{|\xi|^2}{\lambda^2r_j^2}\right) 
\text{ for } \xi\in \C.\]
Then $\chi_j\in C^{\infty}(\C)$ and $\chi_j(\xi)=0$ for $|\xi|=\lambda r_j$. 
Using polar coordinates, one can compute that 
\begin{align*} 
\int_{\{|\xi|<\lambda r_j\}}\chi_j(\xi) dV(\xi)&=1, \\ 
\int_{\{|\xi|<\lambda r_j\}}|\chi_j(\xi)|^2 dV(z)
&=\frac{4}{3\pi\lambda^2r_j^2},\\ 
\int_{\{|\xi|<\lambda r_j\}}|(\chi_j)_{\xi}(\xi)|^2 dV(\xi) 
&=\frac{2}{\pi \lambda^4r_j^4}.
\end{align*}
Let  
\[F_j=f_j d\zb_1\wedge\cdots\wedge d\zb_q \text{ and } 
\Phi =\chi d\zb_1\wedge\cdots\wedge d\zb_q\]  where 
$\chi(z')=\chi_1(z_1)\cdots \chi_q(z_q)$. Assume that $\vartheta$ denotes the 
formal adjoint on $\dbar$. Then 
\[\vartheta \Phi = -\sum_{k=1}^q (-1)^{k-1} 
\frac{\partial\chi}{\partial z_k} d\zb_1\wedge \cdots\wedge 
\widehat{d\zb_k}\wedge \cdots \wedge d\zb_q \]
where $\widehat{d\zb_k}$ means that $d\zb_k$ is missing.  Then 
\begin{align*}
\int_{M_{\lambda}}|\chi_{z_j}(z')|^2dV(z')
=& \int_{\{|\xi|<\lambda r_j\}}|(\chi_j)_{z_j}(\xi)|^2 dV(\xi) 
\prod_{k\neq j}^q\int_{\{|\xi|<\lambda r_k\}}|\chi_{k}(\xi)|^2 dV(\xi)\\
=& \frac{2}{\pi \lambda^4r_j^4}\prod_{k\neq j}^q\frac{4}{3\pi\lambda^2r_k^2}\\
=&\frac{2^{2q-1}}{3^{q-1}\pi^q\lambda^{2q+2}}\frac{1}{r_j^2} \prod_{k=1}^q 
\frac{1}{r_k^2}.
\end{align*}
Then 
\begin{align}\label{Eqn6}
\|\vartheta \Phi \|^2_{M_{\lambda}}
 = \frac{2^{2q-1}}{3^{q-1}\pi^q\lambda^{2q+2}} 
 \frac{\sum_{k=1}^q \frac{1}{r_k^2}}{\prod_{k=1}^q r_k^2}.
\end{align}

Now we will derive an integration by parts formula for $F_j$'s. 
Let $G=gd\zb_1\wedge\cdots\wedge d\zb_q$ where $g$ is a square integrable 
holomorphic function on $\D$. Then 
\[\dbar^*N_qG=\sum_{k=1}^q g_k
d\zb_1\wedge \cdots\wedge \widehat{d\zb_k}\wedge \cdots \wedge 
d\zb_q+H\] 
where $g_k$'s are square integrable functions and $H$ is a $(0,q-1)$-form that 
includes terms containing $d\zb_j$ for some $j\geq q+1$ (in case of $q=1$ the 
form $H$ is zero). Since $G$ is a $\dbar$-closed form $G=\dbar\dbar^*N_qG$. 
Then by comparing the types of forms in $G$ and 
$\dbar\dbar^*N_qG$ we conclude that 
\[\dbar\dbar^*N_qG=\sum_{k=1}^q(-1)^{k-1}
\frac{\partial g_k}{\partial \zb_k} d\zb_1\wedge \cdots\wedge d\zb_q.\] 
Then 
\begin{align*}
 \int_{M_{\lambda}} \langle\vartheta \Phi,  \dbar^*N_qG\rangle 
 = & -\sum_{k=1}^q(-1)^{k-1}\int_{M_{\lambda}}\frac{\partial \chi}{\partial 
z_k}\overline{g_k} \\
=&\sum_{k=1}^q(-1)^{k-1}\int_{M_{\lambda}}\chi \overline{\frac{\partial 
g_k }{\partial \zb_k}}\\
=& \int_{M_{\lambda}}  \langle \Phi,  \dbar\dbar^*N_qG\rangle.
\end{align*}
Now we apply the equality above to $F_j$'s. For a fixed $j$ we have 
\[\int_{M_{\lambda}}  \langle \Phi,  \dbar\dbar^*N_qF_j\rangle
=\int_{M_{\lambda}} \langle\vartheta \Phi,  \dbar^*N_qF_j\rangle.\]
Then (we emphasize the variables in the first line below) using the 
fact that $\int_{M_{\lambda}}\chi(z')dV(z')=1$ in the first equality below we 
get 
\begin{align*}
\overline{f_j(0,z'')}=&\overline{f_j(0,z'')}\int_{M_{\lambda}}\chi(z')dV(z')\\
=&\int_{M_{\lambda}}\chi(z') \overline{f_j(z',z'')}dV(z') \\
 =& \int_{M_{\lambda}} \langle\Phi, \dbar \dbar^*N_qF_j\rangle \\
=&\int_{M_{\lambda}} \langle\vartheta \Phi,  \dbar^*N_qF_j\rangle \\
\leq& \|\vartheta \Phi\|_{M_{\lambda}}\|\dbar^*N_qF_j\|_{M_{\lambda}}.
\end{align*}
We take the norm square of both sides and integrate in $z''$ variables on 
$(1-\lambda)\D_s$ to get 
\begin{align}\label{Eqn5}
\|\dbar^*N_qF_j\| \geq \|\dbar^*N_qF_j\|_{M_{\lambda}\times 
(1-\lambda)\D_s}\geq  \frac{\|f_j\|_{(1-\lambda)\D_s}}{\|\vartheta \Phi 
\|_{M_{\lambda}}}. 
\end{align}

Now we will compute $\|f_j\|_{(1-\lambda)\D_s}$. Let us apply the reproducing 
property of $K_{(1-\lambda)\D_s}(p_j,.)$ to $K_{\D_s}(.,p_j)$ on 
$(1-\lambda)\D_s$. 
\begin{align*}
K_{\D_s}(p_j,p_j)=&\int_{(1-\lambda)\D_s}K_{(1-\lambda)\D_s}(p_j,z'')K_{
\D_s}(z'',p_j)dV(z'')\\
\leq& \|K_{(1-\lambda)\D_s}(p_j,.)\|_{(1-\lambda)\D_s}\|K_{\D_s}(.,p_j)\|_{
(1-\lambda)\D_s}\\
=& \sqrt{K_{(1-\lambda)\D_s}(p_j,p_j)} \|K_{\D_s}(.,p_j)\|_{
(1-\lambda)\D_s}.
\end{align*}
We used the Cauchy-Schwarz inequality for the inequality on the second line. 
Namely, we get  
\[\frac{K_{\D_s}(p_j,p_j)}{K_{(1-\lambda)\D_s}(p_j,p_j)}
\leq \frac{\|K_{\D_s}(.,p_j)\|^2_{(1-\lambda)\D_s}}{K_{\D_s}(p_j,p_j)}.\]
Hence if we write the right hand side above in terms of $f_j$ and use 
\[K_{(1-\lambda)\D_s}(p_j,p_j)
=\frac{1}{(1-\lambda)^{2(n-q)}}K_{\D_s}\left(\frac{p_j}{1-\lambda}, 
\frac{p_j}{1-\lambda}\right)\]
we get 
\begin{align}\label{Eqn7}
\|f_j\|^2_{(1-\lambda)\D_s}
\geq \frac{K_{\D_s}(p_j,p_j)}{K_{(1-\lambda)\D_s}(p_j, p_j) }
=(1-\lambda)^{2(n-q)}\frac{K_{\D_s}(p_j,p_j)}{K_{\D_s}\left(\frac{p_j}{
1-\lambda} , \frac{p_j}{1-\lambda}\right)}. 
\end{align}

Let $\delta>0, \alpha=\frac{\lambda\delta}{1-\lambda}$, and 
$\rho_{\delta}=\delta \rho_0$ where $\rho_0=(0,\ldots,0,1)$. We note that 
$\rho_{\delta}\in \D_s$ for small $\delta>0$. Let us define 
$T_{\alpha}(z)=z+(0,\ldots,0,\alpha)$. Since $\D_s$ is convex, we can choose 
$U$ small enough so that $T_{\alpha}(\D_s\cap U)\subset \D_s$. Since there is 
no analytic disc in the boundary of $\D_s$ through $\rho_0$, 
\cite[Proposition 3.2]{FuStraube98} implies that $\rho_0$ is a peak point  and 
in turn \cite[Theorem 2]{Nikolov02} 
implies that  for $\ep>0$ there exists $\delta_{\ep}>0$ such that 
$0<\delta<\delta_{\ep}$ implies that 
\[K_{\D_s}(\rho_{\delta},\rho_{\delta}) \geq 
(1-\ep)K_{\D_s\cap U}(\rho_{\delta},\rho_{\delta}).\]
Furthermore, we have 
\[K_{\D_s\cap U}(\rho_{\delta},\rho_{\delta})
=K_{T_{\alpha}(\D_s\cap U)}(\rho_{\delta+\alpha},\rho_{\delta+\alpha})
\geq K_{\D_s}(\rho_{\delta+\alpha},\rho_{\delta+\alpha}).\]
Therefore,
\begin{align}\label{Eqn8}
K_{\D_s}(\rho_{\delta},\rho_{\delta}) \geq 
(1-\ep) K_{\D_s}(\rho_{\delta+\alpha},\rho_{\delta+\alpha}). 
\end{align}
Now we choose $\{p_j\}$ as $p_j=\rho_{1/j}=(0,\ldots,0,1/j)$. 
Note that $\delta+\alpha=\delta/(1-\lambda)$. Let us choose $\delta=1/j$. Then 
$\rho_{\delta+\alpha}=\rho_{\delta}/(1-\lambda)=p_j/(1-\lambda)$. 
Furthermore, the fact that $\ep$ is arbitrary and \eqref{Eqn8} imply that  
\[\liminf_{j\to \infty}\frac{K_{\D_s}(p_j,p_j)}{K_{\D_s}\left(\frac{p_j}{
1-\lambda } , \frac{p_j}{1-\lambda}\right)}\geq 1.\]
Then \eqref{Eqn6} and \eqref{Eqn7} imply that 
\begin{align}
\label{Eqn1}
\liminf_{j\to \infty} \frac{\|f_j\|^2_{(1-\lambda)\D_s}}{\|\vartheta \Phi 
\|^2_{M_{\lambda}}} 
 \geq \lambda^{2q+2}(1-\lambda)^{2n-2q} 
  \frac{3^{q-1}\pi^q \prod_{k=1}^qr_k^2}{2^{2q-1} \sum_{k=1}^q \frac{1}{r_k^2}}.
  \end{align}
Now we want to find 
\[\sup\left\{\lambda^{2q+2}(1-\lambda)^{2n-2q}:0\leq \lambda\leq 1\right\}.\] 
One can compute that the maximum of 
$f(\lambda)=\lambda^{2q+2}(1-\lambda)^{2n-2q}$ over the closed interval $[0,1]$ 
is attained at $\lambda=(q+1)/(n+1)$ and it is
\[\frac{(q+1)^{2q+2}(n-q)^{2n-2q}}{(n+1)^{2n+2}}.\]

For the rest of the proof of i. we fix $\lambda =(q+1)/(n+1)$. 
For any $\ep>0$ given we choose a compact operator 
$K_{\ep}:K^2_{(0,q)}(\D)\to L^2_{(0,q)}(\D)$ such that 
\[\|\dbar^*N_q -K_{\ep}\|<\|\dbar^*N_q\|_e+\ep.\] 
Let us choose a subsequence of $\{F_j\}$, if necessary, so that $F_j\to F$ weakly. 
One can show that $F=fd\zb_1\wedge\cdots \wedge d\zb_q$ 
for some $f\in A^2(\D)$. One can also show that $\|F\|\leq 
\liminf_{j\to\infty}\|F_j\|$. Furthermore,   $K_{\ep}(F_j-F)\to 0$ and 
\[\|\dbar^*N_q\|_e\geq \limsup_{j\to \infty} 
\frac{\|\dbar^*N_q(F_j-F)\|}{\|F_j-F\|}-\ep\geq 
 \limsup_{j\to \infty}\frac{\|\dbar^*N_q(F_j-F)\|}{2\|F_j\|}-\ep.\]
 We note that $F|_{(1-\lambda)\D_s}=0$.
This can be seen as follows: Let $K_{\D}$ denote the Bergman kernel of $\D$ 
and  $z\in (1-\lambda)\D_s$. We remind the reader that Lemma \ref{LemWeakConv} 
implies that $\{f_j\}$ converges to zero weakly. Then  
\[F(z)=\langle F,K_{\D}(.,z)\rangle=\lim_{j\to \infty} \langle F_j,K_{\D}(.,z)\rangle
=\lim_{j\to \infty} f_j(z)=0.\]     
Hence $\|f_j-f\|_{(1-\lambda)\D_s}=\|f_j\|_{(1-\lambda)\D_s}$ and using 
\eqref{Eqn4},\eqref{Eqn5},\eqref{Eqn1} we get
\[\|\dbar^*N_q\|^2_e\geq \limsup_{j\to 
\infty}\frac{\|f_j\|^2_{(1-\lambda)\D_s}}{2^2\pi^{q}\tau_{\D}^{2q}
\|\vartheta\Phi\|^2_{M_{\lambda}}}-\ep. \]
Since $\ep$ is arbitrary we get
\[\|\dbar^*N_q\|_e^2 \geq 
\frac{(q+1)^{2q+2}(n-q)^{2n-2q}}{(n+1)^{2n+2}}
\frac{3^{q-1}\prod_{k=1}^qr_k^2}{2^{2q+1} 
\tau_{\D}^{2q}\sum_{k=1}^q \frac{1}{r_k^2}}
=\frac{(c(n,q))^2}{\tau_{\D}^{2q}}\beta_{D(w,r)}^2.\]
This finishes the proof of the first part. 

Now we will prove the case  $q=n-1$.  In this case we denote 
$z=(z',z_n)\in \C^n$ where $z'=(z_1,\ldots,z_{n-1})\in \C^{n-1}$. By 
using translation and rotation if necessary, without loss of generality, we 
may assume that 
\begin{itemize}
 \item[i.] $\D\subset \{z' \in \C^{n-1}:\|z'\|<\tau_{\D}\}\times 
\{z_n\in \C:|z_n|< \tau_{\D}, \re(z_n)>0\}$,
\item[ii.] $M=\{z'\in \C^{n-1}:(z',0)\in b\D\}$ is $(n-1)$-dimensional 
affine variety.
\end{itemize}
Since in this case we assume that $\D$ has $C^1$-smooth boundary for $\ep>0$ 
there exists a wedge 
\[W^{r_0}_{\pi-\ep}=\left\{r\e^{i\theta}\in \C:0\leq 
r<r_0,|\theta|<\frac{\pi-\ep}{2}\right\}\]
such that $M\times W^{r_0}_{\pi-\ep}\subset \D$. 
We choose 
\[f_j(z)=\frac{1}{2^jz_n^{\alpha_i}}\]
where $\alpha_j=1-2^{-2j-1}$. Then $f_j\to 0$ weakly  in $L^2(\D)$ and using i. 
above one can compute that 
\[\|f_j\|^2\leq \frac{1}{2^{2j}} 
\int_{\|z'\|<\tau_{\D}}dV(z') 
\int_{-\pi/2}^{\pi/2}d\theta \int_0^{\tau_{\D}}\frac{dr}{r^{2\alpha_j-1}}
=\pi \omega_{2n-2}\tau_{\D}^{2n-2}\tau_{\D}^{2-2\alpha_j} \] 
where $\omega_{2n-2}$ denotes the volume of the unit ball in 
$\mathbb{R}^{2n-2}$.  We also need to compute $\|f_j\|_{W^{r_0}_{\pi-\ep}}$
\[\|f_j\|^2_{W^{r_0}_{\pi-\ep}}\geq\frac{1}{2^{2j}} 
\int_{-(\pi-\ep)/2}^{(\pi-\ep)/2}d\theta 
\int_0^{r_0}\frac{dr}{r^{2\alpha_j-1}}
=(\pi-\ep)r_0^{2-2\alpha_j}.\] 
Let 
\[F_j=f_j d\zb_1\wedge\cdots\wedge d\zb_{n-1} \text{ and } 
\Phi =\chi d\zb_1\wedge\cdots\wedge d\zb_{n-1}\]  
where $\chi\in C^1(\overline{M})$ not identically zero and $\chi=0$  
on the boundary of $M$. Then for $z_n\in W^{r_0}_{\pi-\ep}$ we have 
 \begin{align*}
 \frac{1}{2^jz_n^{\alpha_j}} \int_{M}\chi(z')dV(z') 
 =& \int_M\chi(z') f_j(z)dV(z') \\ 
 =& \int_M \langle\Phi, \dbar \dbar^*N_{n-1}F_j\rangle \\
=&\int_M \langle\vartheta \Phi,  \dbar^*N_{n-1}F_j\rangle \\
\leq& \|\vartheta \Phi\|_M\|\dbar^*N_{n-1}F_j\|_M.
\end{align*}
We note that, unlike the previous case, in the computations above $\chi$ is not 
necessarily radially symmetric. Then by integrating in the last variable we get 
\begin{align}\label{Eqn9}
\|\dbar^*N_{n-1}F_j\| \geq \|\dbar^*N_{n-1} F_j\|_{M\times W^{r_0}_{\pi-\ep}}
\geq  \frac{\int_M\chi(z')dV(z')}{\left\| \vartheta \Phi 
\right\|_M} \|f_j\|_{W^{r_0}_{\pi-\ep}}. 
\end{align}
Furthermore, the fact that  $\left\| \vartheta \Phi 
\right\|_M=\|\nabla \chi\|_M/2$ and the estimate 
$\|f_j\|^2_{W^{r_0}_{\pi-\ep}}\geq (\pi-\ep)r_0^{2-2\alpha_j}$ imply that 
\begin{align*}
\|\dbar^*N_{n-1}F_j\| \geq \|\dbar^*N_{n-1} F_j\|_{M\times W^{r_0}_{\pi-\ep}}
\geq  \frac{2\int_M\chi(z')dV(z')}{\left\| \nabla \chi \right\|_M} 
\sqrt{(\pi-\ep)r_0^{2-2\alpha_j}}. 
\end{align*}
Finally,
\begin{align*}
\|\dbar^*N_{n-1}\|_e\geq & \limsup_{j\to \infty} 
\frac{\|\dbar^*N_{n-1} F_j\|}{\|F_j\|} \\
\geq & \frac{2\int_M\chi(z')dV(z')}{\left\| \nabla \chi \right\|_M}  
\limsup_{j\to\infty}\sqrt{\frac{(\pi-\ep)r_0^{2-2\alpha_j}}{\pi 
\omega_{2n-2}\tau_{\D}^{2n-2}\tau_{\D}^{2-2\alpha_j}}} \\
=&\frac{2\int_M\chi(z')dV(z')}{\left\| \nabla \chi \right\|_M}  
\sqrt{\frac{\pi-\ep}{\pi \omega_{2n-2}\tau_{\D}^{2n-2}}}.
\end{align*}
Since $\ep$ was arbitrary we get 
\[\|\dbar^*N_{n-1}\|_e\geq \frac{2\int_M\chi(z')dV(z')}{\left\| \nabla \chi 
\right\|_M}  
\frac{1}{\sqrt{\omega_{2n-2}}\tau_{\D}^{n-1}}.\]
Therefore, taking supremum over $\chi$ and using the fact that 
$\omega_{2n-2}=\pi^{n-1}/(n-1)!$ we get 
\[\|\dbar^*N_{n-1}\|_e\geq \alpha_{M}
\sqrt{\frac{(n-1)!}{\pi^{n-1}}}\frac{1}{\tau_{\D}^{n-1}}\]
for every $M$. 
\end{proof}

\begin{remark} 
	The proof of iii. of Theorem \ref{ThmCn}  can be modified to 
work on product domains even though they do not have $C^1$-smooth boundary. 
	Let $U$ be a bounded pseudoconvex domain in $\C^n$ and 
	$\D_r=\{z\in \C:|z|<r\}\times U$. Then the proof of iii. in Theorem 
\ref{ThmCn} modified to work on $\D_r$ with $M=U$ implies the 	following 
essential norm estimate of $N_n$ on $\D_r$:
	\[\|N_n\|_e\geq \alpha_{U}^2 \frac{n!}{\pi^n\tau_{\D_r}^{2n}}.\]
Combining this estimate with H\"{o}rmander's estimate for the norm of $N_n$ 
on $\D_r$  we get  
	\[ \alpha_{U}^2 \frac{n!}{\pi^n\tau_{\D_r}^{2n}}\leq \|N_{n}\|\leq 
\e\frac{\tau_{\D_r}^2}{n}\] 
 Then letting $r$ go to zero (hence $\tau_{\D_r}\to \tau_U$) we get 
	\[\alpha_{U}\leq 
\frac{\tau_{U}^{n+1}}{n}\sqrt{\frac{\e\pi^n}{(n-1)!}}.\]
When $U=\mathbb{D}$ this inequality is not sharp because  
	$4\sqrt{\e\pi}>\alpha_{\mathbb{D}}=\sqrt{\pi/2}$ (see Remark \ref{Rmk2}). 
In case $U$ is a simply connected domain in the complex plane it is known that 
\[\alpha_U\leq \frac{V(U)}{\sqrt{2\pi}}.\] 	
This is known as Saint-Venant’s inequality (see 
\cite[pg 121]{PolyaSzegoBook} and also 
\cite{Makai66,BellFergusonLundberg14,FleemanKhavinson15}).
\end{remark}

Now we are ready to prove Theorem \ref{ThmCn}. 

\begin{proof}[Proof of Theorem \ref{ThmCn}]
The proof of i. follows from Corollary \ref{CorPercolate} and  the fact that if 
the boundary of a bounded convex domain $\D$ does not contain any analytic 
variety of dimension greater than or equal to $q\geq 1$ then $N_q$ on $\D$ is compact 
 \cite[Theorem 1.1]{FuStraube98}. 

To prove ii. let us assume that $1\leq q\leq q_{\D}\leq n-1$. The fact that $\dbar N_q$ 
is compact for $q\geq q_{\D}+1$ together with i. in Proposition \ref{PropCn} and 
the first equation in \eqref{EqnRange} imply that 
\begin{align*}
\|N_{q_{\D}}\|_e=&\|\dbar^*N_{q_{\D}}\|_e^2\\
\geq& \frac{C(n,q_{\D})}{\tau_{\D}^{2q_{\D}}}
\sup\Big\{\beta_{D(w,r)}^2:D(w,r)\text{ is } q_{\D}
\text{-dimensional  polydisc in } b\D \text{ with } r\geq 0 \Big\}
\end{align*}
where 
\[C(n,q_{\D})=\frac{(q_{\D}+1)^{2q_{\D}+2}(n-q_{\D})^{2n-2q_{\D}}}{(n+1)^{2n+2}}
\frac{3^{q_{\D}-1}}{2^{2q_{\D}+1}}.\]
Then Corollary \ref{CorPercolate} implies that  
\[\|N_q\|_e\geq \frac{C(n,q_{\D})}{\tau_{\D}^{2q_{\D}}}
\sup\Big\{\beta_{D(w,r)}^2:D(w,r)\text{ is }
q_{\D}\text{-dimensional  polydisc in } b\D \text{ with } r\geq 0 \Big\}\]
for $1\leq q\leq q_{\D}$. 

The proof of iii. is similar to the proof of ii. The only difference is that we use the 
essential norm estimate for $\dbar^*N_{n-1}$ in ii. in Proposition \ref{PropCn}.
\end{proof}

The following lemma will be used to compute $\alpha_{\D}$ in case $\D$ is an 
annulus in $\C$. 

\begin{lemma} \label{LemTorsionalRigidity}
 Let $\D$ be a $C^1$-smooth bounded domain in $\C$ and $u\in C^2(\D)\cap 
C(\Dc)$ be the real valued function satisfying the following properties: 
$u=0$ on $b\D$ and $u_{z\zb}=-1$ on $\D$. Then 
\[\alpha_{\D}=\frac{\int_{\D} u(z)dV(z)}{\|u_z\|}=\|u_z\|.\]
\end{lemma}
\begin{proof}
 Using the fact that $u$ is real valued together with integration by parts we 
get 
 \[\|u_z\|^2=\int_{\D} u_z(z)u_{\zb}(z)dV(z)
=-\int_{\D}u(z)u_{z\zb}(z)dV(z)=\int_{\D}u(z)dV(z).\]
Also one can check that $\|\nabla u\| = 2\|u_z\|$. Then 
$\alpha_{\D}\geq \frac{\int_{\D} u(z)dV(z)}{\|u_z\|}=\|u_z\|$.

To get the converse. Let $f\in C^2(\D)\cap C(\Dc)$ be a real valued function 
that vanishes on $b\D$ and  $f\not\equiv 0$. Then 
\[\int_{\D} f(z)dV(z)=-\int_{\D} f(z)u_{z\zb}(z)dV(z)
=\int f_z(z)u_{\zb}(z)dV(z)\leq\|f_z\| \|u_z\|.\]
That is, $\frac{\int_{\D} f(z)dV(z)}{\|f_z\|} \leq \|u_z\|$. 
Taking supremum over $f$ we get $\alpha_{\D}\leq \|u_z\|$. 
\end{proof}

\begin{remark}\label{Rmk2}
	Let $\mathbb{D}_r$ be the open disc with radius $r$. Then one can compute 
	$\alpha_{\mathbb{D}_r}=\sqrt{\pi/2}r^2$ because in this case $u(z)=r^2-|z|^2$.
	
	Now let us compute $\alpha_{A_r}$ for the annulus 
	$A_r=\{z\in \C: 1<|z|<r\}$.  The function  
	\[u(z)=r^2-|z|^2-\frac{r^2-1}{\log r} 
	(\log r-\log|z|)\]
	satisfies the conditions in the lemma. That is, $u_{z\zb}=-1$ on $A_r$ and 
	$u=0$ on $bA_r$. Then
	\[u_z=-\zb+\frac{r^2-1}{\log r}\frac{1}{2z}\]
	and one can compute that 
	\[\|u_z\|^2=\frac{\pi}{2}\left(r^4-1-\frac{(r^2-1)^2}{\log r}\right).\]
	We note that this is $P'$ in \cite[pg 103]{PolyaSzegoBook}. Therefore  
	Lemma \ref{LemTorsionalRigidity} implies that 
	\begin{align}\label{EqnAnnulus}
	\alpha_{A_r}=\sqrt{\frac{\pi}{2}\left(r^4-1-\frac{(r^2-1)^2}{\log r}\right)}.
	\end{align}  
\end{remark}

Now we are ready to prove Theorem \ref{ThmWorm}. 

\begin{proof}[Proof of Theorem \ref{ThmWorm}]
Let us denote $A_a^b=\{\xi\in \C:a<|\xi|<b\}$ and assume that 
$1<\eta <\min\{\e^{\pi/2\beta},r\}$. Since $2\beta\log \eta \in(0,\pi)$ 
for every $0<\ep<\pi-2\beta\log\eta$ there exists $\delta>0$ such that 
we can put a  wedge with angle $\pi-2\beta \log\eta-\ep$ and radius $\delta>0$ 
in the domain that is perpendicular to $A_{1+\ep}^{\eta-\ep}$. More precisely, 
\[W^{\delta}_{\pi-2\beta\log \eta-\ep}\times A_{1+\ep}^{\eta-\ep}
\subset \D_{\beta,r}\cap \{(z_1,z_2)\in \C^2:1<|z_2|<\eta\}\] 
We can also put $\D_{\beta,r}\cap \{(z_1,z_2)\in \C^2:1<|z_2|<\eta\}$ 
in a similar product space. Hence we have 
\[W^{\delta}_{\pi-2\beta\log \eta-\ep}\times A_{1+\ep}^{\eta-\ep}\subset 
\D_{\beta,r}\cap \{(z_1,z_2)\in \C^2:1<|z_2|<\eta\} 
\subset W^{2}_{\pi+2\beta \log \eta}\times A_{1}^{\eta}\]
We use the same sequence of functions 
\[f_j(z_1,z_2)=\frac{1}{2^{j}z_1^{\alpha_j}}\] 
as in the proof of Theorem \ref{ThmCn}. Let  $\chi_{\eta}$ be a function 
independent of $z_1$ such that 
$\chi_{\eta}(z_2)=1$ if $1\leq |z_2|\leq \eta$ and 
$\chi_{\eta}(z_2)=0$ otherwise. Then  
\[\|f_j\|^2_{W^{\delta}_{\pi-2\beta\log \eta-\ep}}
=(\pi-2\beta\log \eta-\ep)\delta^{2-2\alpha_j} \text{ and } 
\|\chi_{\eta}f_j\|^2\leq \pi(\pi+2\beta \log \eta)(\eta^2-1)2^{2-2\alpha_j}.\]

Let $\chi\in C^1(\overline{A_{1+\ep}^{\eta-\ep}})$ be a real valued function such that 
$\chi\equiv 0$ on the boundary of $A_{1+\ep}^{\eta-\ep}$ and 
$\int_{A_{1+\ep}^{\eta-\ep}}\chi(z_2)dV(z_2)=1$. We think of $\chi$ as 
a function of $z_2$. Then by similar computations as in 
\eqref{Eqn9} for $F_j=\chi_{\eta}f_jd\zb_2$ we get  
\[\|\dbar^*N_1(\chi_{\eta}f_jd\zb_2)\|\geq 
\frac{2\|f_j\|_{W^{\delta}_{\pi-2\beta\log \eta-\ep}}}{\|\nabla\chi\|_{A_{1+\ep}^{\eta-\ep}}}
=\frac{2}{\|\nabla\chi\|_{A_{1+\ep}^{\eta-\ep}}}\delta^{1-\alpha_j} 
\sqrt{\pi-2\beta\log \eta-\ep}.\]
Then 
\begin{align*} 
\|\dbar^*N_1\|_e\geq & 
\limsup_{j\to \infty}\frac{\|\dbar^*N_1(\chi_{\eta}f_jd\zb_2)\|}{\|\chi_{\eta}f_j\|} \\ 
\geq & \limsup_{j\to\infty}  \frac{2}{\|\nabla\chi\|_{A_{1+\ep}^{\eta-\ep}}} 
\frac{\sqrt{\pi-2\beta\log \eta-\ep}}{\sqrt{\pi(\pi+2\beta\log \eta)(\eta^2-1)}}
\left(\frac {\delta}{2}\right)^{1-\alpha_j}\\
= &\frac{2}{\|\nabla\chi\|_{A_{1+\ep}^{\eta-\ep}}} 
\frac{\sqrt{\pi-2\beta\log \eta-\ep}}{\sqrt{\pi(\pi+2\beta\log \eta)(\eta^2-1)}}.
\end{align*}
Therefore, if we let $\ep\to 0$ and take supremum over $\chi$ we get 
\[\|\dbar^*N_1\|_e\geq \alpha_{A_{\eta}}
\frac{\sqrt{\pi-2\beta\log\eta}}{\sqrt{\pi(\pi+2\beta\log \eta)(\eta^2-1)}}.\]
Using the fact that $\|N_1\|_e=\|\dbar^*N_1\|_e^2$ on domains in $\C^2$ 
we get 
\[\|N_1\|_e\geq \frac{\alpha_{A_{\eta}}^2}{\pi (\eta^2-1)}
\frac{\pi-2\beta\log\eta}{\pi+2\beta \log \eta}.\]
Then \eqref{EqnAnnulus} implies that 
\[\|N_1\|_e\geq \left(\frac{\eta^2+1}{2}-\frac{\eta^2-1}{2\log \eta}\right)
\frac{\pi-2\beta\log \eta}{\pi+2\beta\log \eta}.\]
We complete the proof by taking maximum of the left hand side for 
$1<\eta<\min\{\e^{\pi/\beta},r\}$. 
\end{proof}

\section*{Acknowledgment}
We would like to thank Akaki Tikaradze for giving us the idea for the 
proof of Lemma \ref{LemAkaki}, Trieu Le for pointing out a mistake in 
an earlier version, and the referee for suggestions that has 
improved the paper.


\end{document}